\documentclass[11pt]{article}
\usepackage{amssymb,amsmath}
\usepackage[mathscr]{eucal}
\usepackage[cm]{fullpage}
\usepackage[english]{babel}
\usepackage[latin1]{inputenc}
\usepackage{xcolor}
\usepackage{graphicx,tikz}

\def\dom{\mathop{\mathrm{Dom}}\nolimits}
\def\im{\mathop{\mathrm{Im}}\nolimits}
\def\rank{\mathop{\mathrm{rank}}\nolimits}
\def\d{\mathrm{d}} 
\def\id{\mathrm{id}}
\def\N{\mathbb N}
\def\PT{\mathcal{PT}}
\def\T{\mathcal{T}}
\def\Sym{\mathcal{S}}
\def\DP{\mathcal{DP}}

\def\D{\mathcal{D}}
\def\ODP{\mathcal{ODP}} 
\def\DI{\mathcal{DI}}

\def\I{\mathcal{I}}
\def\DPW{\mathcal{DPW}}
\newcommand{\<}{\leqslant} 
\renewcommand{\>}{\geqslant} 
\newtheorem{theorem}{Theorem}[section]
\newtheorem{proposition}[theorem]{Proposition}
\newtheorem{corollary}[theorem]{Corollary}
\newtheorem{lemma}[theorem]{Lemma}
\newenvironment{proof}{\begin{trivlist}\item[\hskip%
\labelsep{\bf Proof.}]}%
{\qed\rm\end{trivlist}}
\newcommand{\qed}{{\unskip\nobreak
\hfil\penalty50\hskip .001pt \hbox{}
          \nobreak\hfil
         \vrule height 1.2ex width 1.1ex depth -.1ex
           \parfillskip=0pt\finalhyphendemerits=0\medbreak}}

\newcommand{\lastpage}{\addresss}

\newcommand{\addresss}{\small \sf

\noindent{\sc V\'\i tor H. Fernandes},
Center for Mathematics and Applications (NovaMath)
and Department of Mathematics, NOVA FCT, 
Universidade Nova de Lisboa,
Monte da Caparica,
2829-516 Caparica,
Portugal;
e-mail: vhf@fct.unl.pt.
}

\title{On the monoid of partial isometries of a wheel graph} 

\author{V\'\i tor H. Fernandes\footnote{This work is funded by national funds through the FCT - Funda\c c\~ao para a Ci\^encia e a Tecnologia, I.P., under the scope of the projects UIDB/00297/2020 and UIDP/00297/2020 (NovaMath - Center for Mathematics and Applications).}}

\begin{document}

\maketitle

\vspace*{-.75cm}

\begin{abstract}
In this paper, we consider the monoid $\DPW_n$ of all partial isometries of a wheel graph $W_n$ with $n+1$ vertices. 
Our main objective is to determine the rank of $\DPW_n$. 
In the process, we also compute the ranks of three notable subsemigroups of $\DPW_n$. 
We also describe Green's relations of $\DPW_n$ and of its three considered subsemigroups.
\end{abstract}

\medskip

\noindent{\small 2020 \it Mathematics subject classification: \rm 20M10, 20M20, 05C12, 05C25.}

\noindent{\small\it Keywords: \rm transformations, orientation, partial isometries, wheel graphs, rank.}

\section*{Introduction and preliminaries}

Let $\Omega$ be a set. Denote by $\PT(\Omega)$ the monoid (under composition) of all
partial transformations on $\Omega$, by $\T(\Omega)$ the submonoid of $\PT(\Omega)$ of all
full transformations on $\Omega$, by $\I(\Omega)$
the \textit{symmetric inverse monoid} on $\Omega$, i.e.
the inverse submonoid of $\PT(\Omega)$ of all
partial permutations on $\Omega$,
and by $\Sym(\Omega)$ the \textit{symmetric group} on $\Omega$,
i.e. the subgroup of $\PT(\Omega)$ of all
permutations on $\Omega$. 
Let $n\in\N$. 
If $\Omega$ is a finite set with $n$ elements, 
say $\Omega=\Omega_n=\{1,2,\ldots,n\}$, we denote, as usual,
$\PT(\Omega)$, $\T(\Omega)$, $\I(\Omega)$ and $\Sym(\Omega)$ simply by $\PT_n$, $\T_n$, $\I_n$ and $\Sym_n$, respectively.

\smallskip

Recall that the \textit{rank} of a monoid $M$ is the minimum size of a generating set of $M$, i.e.
the minimum of the set $\{|X|\mid \mbox{$X\subseteq M$ and $X$ generates $M$}\}$. 
For $n\geqslant3$, 
it is well-known that $\Sym_n$
has rank $2$ 
and
$\T_n$, $\I_n$ and $\PT_n$ have
ranks $3$, $3$ and $4$, respectively.
The survey \cite{Fernandes:2002survey} presents
these results and similar ones for other classes of transformation monoids,
in particular, for monoids of order-preserving transformations and
for some of their extensions.
For example, the rank of the extensively studied monoid of all order-preserving transformations of a chain with $n$ elements is $n$,
which is a result proved by Gomes and Howie \cite{Gomes&Howie:1992} in 1992.
More recently, for instance, the papers
\cite{
Araujo&al:2015,
Dimitrova&al:2020, 
Dimitrova&Koppitz:2017, 
Fernandes&al:2014,
Fernandes&Quinteiro:2014,
Fernandes&Sanwong:2014}
are dedicated to the computation of the ranks of certain classes of transformation semigroups or monoids.

\smallskip

Let $G=(V,E)$ be a finite simple connected graph. 
The (\textit{geodesic}) \textit{distance} between two vertices $x$ and $y$ of $G$, 
denoted by $\d_G(x,y)$, is the length of a shortest path between $x$ and $y$, i.e. the number of edges in a shortest path between $x$ and $y$. 
We say that $\alpha\in\PT(V)$ is a \textit{partial isometry} or \textit{distance preserving partial transformation} of $G$ if
$
\d_G(x\alpha,y\alpha) = \d_G(x,y) 
$
for all $x,y\in\dom(\alpha)$. 
Let $\DP(G)=\{\alpha\in\PT(V)\mid\mbox{$\alpha$ is a partial isometry of $G$}\}$. 
Then, clearly, $\DP(G)$ is a submonoid of $\PT(V)$ and, 
as a consequence of the property
$\d_G(x,y)=0$ if and only if $x=y$,   
for all $x,y\in V$, it immediately follows that $\DP(G)\subseteq\I(V)$. Moreover, 
$\DP(G)$ is an inverse submonoid of $\I(V)$
(see \cite{Fernandes&Paulista:2023}).
For instance, if $G=(V,E)$ is a complete graph, i.e. $E=\{\{x,y\}\mid x,y\in V, x\neq y\}$, then $\DP(G)=\I(V)$.

For $n\>1$, consider the undirected path $P_n$ with $n$ vertices, i.e.
$$
P_n=\left(\Omega_n,\{\{i,i+1\}\mid 1\<i\<n-1\}\right).
$$
Then, $\DP(P_n)$ coincides with the monoid
$
\DP_n=\{\alpha\in\I_n \mid |i\alpha-j\alpha|=|i-j| \mbox{~for all $i,j\in\dom(\alpha)$}\}
$
of all partial isometries on $\Omega_n$. 
The study of partial isometries on $\Omega_n$ was initiated
by Al-Kharousi et al.~in \cite{AlKharousi&Kehinde&Umar:2014} and \cite{AlKharousi&Kehinde&Umar:2016}.
The first of these two papers is dedicated to investigating some combinatorial properties of
the monoid $\DP_n$ and of its submonoid $\ODP_n$ of all order-preserving (considering the usual order on $\Omega_n$) partial isometries, 
in particular, to computing  their cardinalities. The second one presents the study of some of their algebraic properties, namely of Green's relations and ranks. 
Presentations for both the monoids $\DP_n$ and $\ODP_n$ were given by Fernandes and Quinteiro in \cite{Fernandes&Quinteiro:2016} 
and the maximal subsemigroups of $\ODP_n$ were characterized by Dimitrova in \cite{Dimitrova:2013}. 
Moreover, for $2 \leqslant r \leqslant n-1$, Bugay et al.~obtained in \cite{Bugay&Yagci&Ayik:2018} the ranks of the subsemigroups
$\DP_{n,r} = \{\alpha\in\DP_n\mid |\im(\alpha)|\leqslant r\}$ of $\DP_n$ and
$\ODP_{n,r} = \{\alpha\in\ODP_n\mid |\im(\alpha)| \leqslant r\}$ of $\ODP_n$. 

The monoid $\DP(S_n)$ of all partial isometries of a star graph $S_n$ with $n$ vertices ($n\geqslant1$), i.e. 
$$
S_n=\left(\{0,1,\ldots,n-1\},\{\{0,i\}\mid 1\<i\<n-1\}\right), 
$$
was considered by Fernandes and Paulista in \cite{Fernandes&Paulista:2023}. 
They determined the rank and size of $\DP(S_n)$ as well as described its Green's relations. 
A presentation for $\DP(S_n)$ was also exhibited in \cite{Fernandes&Paulista:2023}.

For $n\geqslant3$, consider the cycle graph
$$
C_n=(\Omega_n, \{\{i,i+1\}\mid 1\<i\<n-1\}\cup\{\{1,n\}\})
$$
with $n$ vertices.  
The monoid $\DP(C_n)$ of all partial isometries of $C_n$ was studied by Fernandes and Paulista in \cite{Fernandes&Paulista:2022arxiv}. 
They showed that $\DP(C_n)$ 
coincides with the inverse submonoid $\DI_n$ of $\I_n$ 
formed by all restrictions of (the elements of) a dihedral subgroup of $\Sym_n$ of order $2n$ and called it the dihedral inverse monoid on $\Omega_n$. 
In \cite{Fernandes&Paulista:2022arxiv}, the authors also determined the cardinal and rank of $\DI_n$ as well as presentations and descriptions of its Green's relations.

\smallskip

Now, for a positive integer $n$, 
denote by $W_n$ a \textit{wheel graph} with $n+1$ vertices and fix 
$$
W_n=\big(\Omega_n^0,\{\{0,i\}\mid 1\leqslant i\leqslant n\}\cup\{\{i,i+1\}\mid 1\leqslant i\leqslant n-1\}\cup\{\{1,n\}\}\big), 
$$ 
where $\Omega_n^0=\{0,1,\ldots,n\}$. 
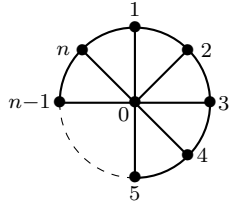
\begin{center}
\begin{tikzpicture}
\draw (0,0) node{$\bullet$} (0,2) node{$\bullet$} (-1,1) node{$\bullet$} (1,1) node{$\bullet$}; 
\draw (0.7,0.3) node{$\bullet$} (0.7,1.7) node{$\bullet$} (-0.7,1.7) node{$\bullet$}; 
\draw (0,-0.2) node{$\scriptstyle5$} (0,2.25) node{$\scriptstyle1$} (-1.4,1) node{$\scriptstyle n-1$} (1.18,1) node{$\scriptstyle3$}; 
\draw (0.9,0.3) node{$\scriptstyle4$} (0.95,1.7) node{$\scriptstyle2$} (-0.95,1.7) node{$\scriptstyle n$}; 
\draw (0,1) node{$\bullet$}; \draw (-.15,.85) node{$\scriptstyle0$}; 
\draw[thick] (0,1) -- (0,0); \draw[thick] (0,1) -- (0,2); \draw[thick] (0,1) -- (-1,1); \draw[thick] (0,1) -- (1,1); 
\draw[thick] (0,1) -- (0.7,0.3); \draw[thick] (0,1) -- (0.7,1.7) ; \draw[thick] (0,1) -- (-0.7,1.7); 
\draw[thick] (0,0) arc (-90:180:1); 
\draw[thin,dashed] (0,0) arc (-90:-180:1);
\end{tikzpicture}
\end{center}

For convenience, on several occasions throughout this paper, 
for elements of $\Omega_n$, when they are playing the role of vertices of $W_n$,  
we take addition (or subtraction) modulo $n$, with $\Omega_n$ as set of representatives. 
For instance, when $1$ and $n$ are considered as vertices of $W_n$, the expressions $n+1$ and $1-n$ also denote the vertex $1$ of $W_n$.  
Notice that, in this line, we can write simply $W_n=\big(\Omega_n^0,\{\{0,i\},\{i,i+1\}\mid 1\leqslant i\leqslant n\}\big)$.  

\smallskip 

Let  $\d=\d_{W_n}$. Then, clearly, we have: 
\begin{itemize}
\item $\d(0,i)=1$ for $1\leqslant i\leqslant n$; 
\item $\d(1,n)=\d(i,i+1)=1$ for $1\leqslant i\leqslant n-1$;
\item $\d(i,n)=2$ for $2\leqslant i\leqslant n-2$; and 
\item $\d(i,j)=2$ for $1\leqslant i< j-1\leqslant n-2$. 
\end{itemize}

Let us denote $\DP(W_n)$ simply by $\DPW_n$. Observe that $\DPW_n$ is an inverse submonoid of $\I(\Omega_n^0)\simeq\I_{n+1}$. 

\smallskip 

In this paper, we study the monoid $\DPW_n$ of all partial isometries of a wheel graph $W_n$ with $n+1$ vertices. 
Our main objective is to determine the rank of $\DPW_n$, which we obtain in Section \ref{gr}, our last section.  
In the process, we also compute the ranks of three notable subsemigroups of $\DPW_n$, 
namely $\DPW_n^-$, $\DPW_n^+$ and $\DPW_n^-\cup\DPW_n^+$, that we define in Section \ref{subs}. 
Green's relations of $\DPW_n^-$, $\DPW_n^+$, $\DPW_n^-\cup\DPW_n^+$ and $\DPW_n$ are described in Section \ref{green}. 
Constituting key results to achieve the main result of this paper, 
in Section \ref{gr-}, we determine a set of generators and the rank of $\DPW_n^-$. 

\smallskip 

It is clear that $\DPW_1\simeq\I_{2}$, $\DPW_2\simeq\I_{3}$ and $\DPW_3\simeq\I_{4}$.  
Therefore, we focus our attention on $\DPW_n$ only for $n\geqslant 4$. 
Thus, until the end of this paper, we consider $n\geqslant4$.

\medskip 

For general background on Semigroup Theory and standard notations, we refer to Howie's book \cite{Howie:1995}.

\smallskip

We would like to point out that we made considerable use of computational tools, namely GAP \cite{GAP4}.

\section{The monoids $\DPW_n^-$ and $\DPW_n^+$}\label{subs}

In this section, we start by presenting some basic properties of $\DPW_n$ which, in particular, lead us to define the monoids $\DPW_n^-$ and $\DPW_n^+$. 
We also give characterizations of the elements of $\DPW_n^-$ and determine the groups of units of these monoids. 

\smallskip 

Let $J_k=\{\alpha\in\DPW_n\mid \rank(\alpha)=k\}$ for $0\< k\<n+1$. 

\begin{lemma}\label{split}
Let $0\<k\<n+1$ and $\alpha\in J_k$.  Then: 
\begin{enumerate}
\item If $0\in\dom(\alpha)$ and $0\alpha\neq0$ then $k\<4$;
\item If $0\in\im(\alpha)$ and $0\alpha^{-1}\neq0$ then $k\<4$;
\item If $k\geqslant4$ then $0\in\dom(\alpha)$ if and only if $0\in\im(\alpha)$; 
\item If $k\geqslant5$ and $0\in\dom(\alpha)\cup\im(\alpha)$ then $0\in\dom(\alpha)\cap\im(\alpha)$ and $0\alpha=0$. 
\end{enumerate}
\end{lemma}
\begin{proof} 
Let 
$
\alpha=\left(\begin{smallmatrix} 
i_1&i_2&\cdots&i_k\\
j_1&j_2&\cdots&j_k
\end{smallmatrix}\right)  
$
with $0\leqslant i_1<i_2<\cdots<i_k\leqslant n$. 

1. Suppose that $i_1=0$ and $j_1\neq0$. Then $\d(j_1,j_t)=\d(0,i_t)=1$ and so $j_t\in\{0,j_1-1,j_1+1\}$, 
for $2\< t\< k$. Hence, $\{j_2,\ldots,j_k\}\subseteq\{0,j_1-1,j_1+1\}$ and, by the injectivity of $\alpha$, we have $k-1\<3$, i.e. $k\<4$. 

Observe that, in this case, if $k=4$ then $\{j_2,j_3,j_4\}=\{0,j_1-1,j_1+1\}$ and so $0\in\im(\alpha)$. 

\smallskip 

2. To prove Property 2, it suffices to apply Property 1 to $\alpha^{-1}$. 

\smallskip 

3. Suppose that $k\>4$ and $0\in\dom(\alpha)$. If $0\alpha=0$ then $0\in\im(\alpha)$. 
On the other hand, if $0\alpha\neq0$ then $k\<4$, by Property 1, and so $k=4$. 
Hence, as observed above, $0\in\im(\alpha)$. 

If $k\>4$ and $0\in\im(\alpha)$ then $0\in\dom(\alpha^{-1})$ and so, from what we just proved, $0\in\im(\alpha^{-1})$, i.e. $0\in\dom(\alpha)$. 

\smallskip 

4. Suppose that $k\geqslant5$ and $0\in\dom(\alpha)\cup\im(\alpha)$. Then, by Property 3, we obtain $0\in\dom(\alpha)\cap\im(\alpha)$ and so, 
by Property 1, we have $0\alpha=0$, as required.  
\end{proof}

Given a set $\Omega$, a subset $X$ of $\Omega$ and $\alpha\in\PT(\Omega)$, we denote the restriction of $\alpha$ to $X$ by $\alpha|_X$. 
So, $\alpha|_X$ is the partial transformation of $\Omega$ such that $\dom(\alpha|_X)=\dom(\alpha)\cap X$ and $x\alpha|_X=x\alpha$ for $x\in\dom(\alpha|_X)$. 

\medskip 

Let 
$$
\DPW_n^-=\{\alpha\in\DPW_n\mid 0\not\in\dom(\alpha)\cup\im(\alpha)\}
$$ 
and 
$$
\DPW_n^+=\{\alpha\in\DPW_n\mid \mbox{$0\in\dom(\alpha)$ and $0\alpha=0$}\}. 
$$
As a consequence of Lemma \ref{split}, it is easy to deduce that $\DPW_n^+$ and $\DPW_n^-$ are isomorphic inverse subsemigroups of $\DPW_n$. 
An isomorphism $\Psi$ from $\DPW_n^+$ into $\DPW_n^-$ can be defined by $\alpha\Psi=\alpha|_{\Omega_n}$ for $\alpha\in\DPW_n^+$. 
Moreover, $\DPW_n^+$ is a submonoid of $\DPW_n$ and $\DPW_n^-$ is a monoid but not a submonoid of $\DPW_n$. 
On the other hand, we may consider $\DPW_n^-$ contained in $\I_n$ and, therefore, it is an inverse submonoid of $\I_n$. 
Observe that it is easy to conclude that $\DPW_n^-\cup\DPW_n^+$ is also an inverse submonoid of $\DPW_n$ which admits $\DPW_n^-$ as an ideal. 
Notice also that, if $\alpha\in\DPW_n\setminus(\DPW_n^-\cup\DPW_n^+)$ then $0\in\dom(\alpha)\cup\im(\alpha)$ and, 
by Properties 1 and 2 of Lemma \ref{split}, we have $\rank(\alpha)\<4$.  

\smallskip 

Next, consider the following permutations of $\Omega_n$ of order $n$ and $2$, respectively:
$$
g=\begin{pmatrix}
1&2&\cdots&n-1&n\\
2&3&\cdots&n&1
\end{pmatrix}
\quad\text{and}\quad
h=\begin{pmatrix}
1&2&\cdots&n-1&n\\
n&n-1&\cdots&2&1
\end{pmatrix}.
$$

It is clear that $g,h\in\DPW_n$.
Moreover, for $n\geqslant3$, $g$ and $h$ generate the well-known \textit{dihedral group} $\D_{2n}$ of order $2n$ 
(considered as a subgroup of $\Sym_n$). In fact, for $n\geqslant3$, 
$$
\D_{2n}=\langle g,h\mid g^n=1,h^2=1, hg=g^{n-1}h\rangle=\{1,g,g^2,\ldots,g^{n-1}, h,hg,hg^2,\ldots,hg^{n-1}\} 
$$
and we have 
$$
g^k=\begin{pmatrix} 
1&2&\cdots&n-k&n-k+1&\cdots&n\\
1+k&2+k&\cdots&n&1&\cdots&k
\end{pmatrix}, 
\quad\text{i.e.}\quad  
ig^k=\left\{\begin{array}{ll}
i+k & \mbox{if $1\leqslant i\leqslant n-k$}\\
i+k-n & \mbox{if $n-k+1\leqslant i\leqslant n$,}  
\end{array}\right.
$$
and 
$$
hg^k=\begin{pmatrix} 
1&\cdots&k&k+1&\cdots&n\\
k&\cdots&1&n&\cdots&k+1
\end{pmatrix}, 
\quad\text{i.e.}\quad  
ihg^k=\left\{\begin{array}{ll}
k-i+1 & \mbox{if $1\leqslant i\leqslant k$}\\
n+k-i+1 & \mbox{if $k+1\leqslant i\leqslant n$,} 
\end{array}\right.
$$
for $0\leqslant k\leqslant n-1$. 

\smallskip

Now, recall that, for $\alpha\in\I_n$, we have $\alpha\in\DI_n$ if and only if $\alpha=\sigma|_{\dom(\alpha)}$ for some $\sigma\in\D_{2n}$  
(see \cite{Fernandes&Paulista:2022arxiv}). Observe that, clearly, $g,h\in\DPW_n^-$ and, since $\DPW_n$ (in fact, $\DP(G)$, for any graph $G$) contains all restrictions of each of its elements and, consequently, the same property is valid for $\DPW_n^-$, we get $\DI_n\subseteq\DPW_n^-$. 

\smallskip 

For $1\<i,j\<n$, we define the \textit{arc} $A_{i,j}$ of $\Omega_n$ to be the set 
$$
A_{i,j}=\left\{\begin{array}{ll}
\{i,i+1,\ldots,j-1,j\} & \mbox{if $i\<j$}\\
\{i,\ldots,n,1,\ldots,j\} & \mbox{if $j<i$}. 
\end{array}\right. 
$$ 
Notice that $A_{1,n}=A_{i+1,i}=\Omega_n$ for $1\<i\<n-1$. Observe also that, for $1\<i\<j\<n$, the arc $A_{i,j}$ is also an interval of $\Omega_n$  (for its usual linear order). 

Let $X$ be a subset of $\Omega_n$ and $1\<i,j\<n$. 
We say that $A_{i,j}$ is an \textit{arc of $X$} of $\Omega_n$ if $A_{i,j}\subseteq X$. 
An arc $A_{i,j}$ of $X$ not properly contained in any other arc of $X$ is called a \textit{maximal arc of $X$}.

These concepts allow us to present the following characterizations of $\DPW_n^-$. 

\begin{proposition}\label{char}
Let $\alpha\in\I_n$. Then $\alpha\in\DPW_n^-$ if and only if
\begin{enumerate}
\item $\alpha$ maps maximal arcs of $\dom(\alpha)$ onto maximal arcs of $\im(\alpha)$, and 
\item $\alpha|_{A_{i,j}}\in\DI_n$, for any arc $A_{i,j}$ of $\dom(\alpha)$ with $1\<i,j\<n$  
\end{enumerate}
if and only if
\begin{enumerate}
\item[1'.] $\alpha$ maps maximal arcs of $\dom(\alpha)$ onto maximal arcs of $\im(\alpha)$, and 
\item[2'.] $\alpha|_{A_{i,j}}\in\DI_n$, for any maximal arc $A_{i,j}$of $\dom(\alpha)$ with $1\<i,j\<n$. 
\end{enumerate}
\end{proposition}
\begin{proof} 
First, suppose that $\alpha\in\DPW_n^-$. 

Observe that, for $1\<k\<n$ such that $k,k+1\in\dom(\alpha)$, $\d(k\alpha,(k+1)\alpha)=\d(k,k+1)=1$ and so we have $(k+1)\alpha\in\{k\alpha-1,k\alpha+1\}$. 

Let $1\<i,j\<n$ and suppose that $A_{i,j}$ is an arc of $\dom(\alpha)$. 
Since $\alpha$ is injective, by the previous observation, it is easy to deduce that, 
if $i\leqslant j$ then 
$$
\alpha|_{A_{i,j}}=
\begin{pmatrix} 
i&i+1&\cdots&j\\ 
i\alpha&i\alpha+1&\cdots&i\alpha+j-i
\end{pmatrix}
=g^{i\alpha-i}|_{A_{i,j}}
$$
or
$$
\alpha|_{A_{i,j}}=\begin{pmatrix} 
i&i+1&\cdots&j\\ 
i\alpha&i\alpha-1&\cdots&i\alpha-j+i
\end{pmatrix}
=hg^{i\alpha+i-1}|_{A_{i,j}},
$$ 
and if $j<i$ then 
$$
\alpha|_{A_{i,j}}=
\begin{pmatrix} 
i&\cdots&n&1&\cdots&j\\ 
i\alpha&\cdots&i\alpha+n-i&i\alpha+n-i+1&\cdots&i\alpha+n-i+j
\end{pmatrix}
=g^{i\alpha-i}|_{A_{i,j}}
$$
or
$$
\alpha|_{A_{i,j}}=\begin{pmatrix} 
i&\cdots&n&1&\cdots&j\\ 
i\alpha&\cdots&i\alpha-n+i&i\alpha-n+i-1&\cdots&i\alpha-n+i-j
\end{pmatrix}
=hg^{i\alpha+i-1}|_{A_{i,j}}. 
$$ 
Thus, $\alpha|_{A_{i,j}}\in\DI_n$. 
Moreover, we can also conclude that $\alpha$ maps arcs of $\dom(\alpha)$ onto arcs of $\im(\alpha)$. 
Since we took an arbitrary $\alpha\in\DPW_n^-$, we can also conclude that  $\alpha^{-1}$ maps arcs of $\dom(\alpha^{-1})=\im(\alpha)$ 
onto arcs of $\im(\alpha^{-1})=\dom(\alpha)$ and, therefore, it is easy to deduce that $\alpha$ maps maximal arcs of $\dom(\alpha)$ onto maximal arcs of $\im(\alpha)$. 

Hence, we have proved Conditions 1 and 2. Since Conditions 1 and 2 imply trivially Conditions 1' and 2', 
it remains to show that these last two conditions imply that $\alpha\in\DPW_n^-$. 

\smallskip 

So, let us suppose that Conditions 1' and 2' are true.  
Let $x,y\in\dom(\alpha)$ and, without loss of generalization, suppose that $x<y$. 
If $x$ and $y$ belong to the same maximal arc $A$ of $\dom(\alpha)$ then $\d(x,y)=\d(x\alpha,y\alpha)$, since $\alpha|_A\in\DI_n\subseteq\DPW_n^-$. 
On the other hand, suppose that $x$ and $y$ belong to distinct maximal arcs of $\dom(\alpha)$. Then, $1<y-x<n-1$ and so $\d(x,y)=2$. 
Suppose, by contradiction, that $\d(x\alpha,y\alpha)=1$. Then $y\alpha=x\alpha-1$ or $y\alpha=x\alpha+1$ or $\{x\alpha,y\alpha\}=\{1,n\}$ and so 
$x\alpha$ and $y\alpha$ must belong to the same maximal arc $B$ of $\im(\alpha)$. Let $A$ be the maximal arc of $\dom(\alpha)$ containing $x$. 
As $\alpha$ maps maximal arcs of $\dom(\alpha)$ onto maximal arcs of $\im(\alpha)$, we have $A\alpha=B$. Then $y\alpha\in B=A\alpha$ and so $y\in A$, 
which is a contradiction. Thus, $\d(x\alpha,y\alpha)=2=\d(x,y)$ and so $\alpha\in\DPW_n^-$, as required.  
\end{proof}

Observe that, from Proposition \ref{char}, it is easy to deduce that $\DPW_4^-=\DI_4$ and $\DPW_5^-=\DI_5$. 

\smallskip 

Now, notice that, obviously, $J_0=\{\emptyset\}$ and $J_{n+1}$ is the group of units of $\DPW_n$. 
Moreover, $J_{n+1}$ is also the group of units of $\DPW_n^+$ and 
$J_n^-=\{\alpha\in J_n\mid 0\not\in\dom(\alpha)\}=\{\alpha\in\DPW_n^-\mid\rank(\alpha)=n\}$ is the group of units of $\DPW_n^-$. 
Furthermore, as a consequence of Proposition \ref{char} and regarding that $J_n^-=J_{n+1}\Psi$, 
where $\Psi: \DPW_n^+\longrightarrow\DPW_n^-$ is the isomorphism defined above, we immediately have the following corollary. 

\begin{corollary}\label{gpu}
One has $J_n^-=\D_{2n}$ and $J_{n+1}\simeq\D_{2n}$.  
\end{corollary}

\section{Green's relations}\label{green}

In this section we will describe Green's relations of $\DPW_n^-$, $\DPW_n^+$, $\DPW_n^-\cup\DPW_n^+$ and $\DPW_n$. 
Remember that, given a set $\Omega$ and an inverse submonoid $M$ of $\I(\Omega)$, it is well known that 
the Green's relations $\mathscr{L}$, $\mathscr{R}$ and $\mathscr{H}$
of $M$ can be described as follows: for $\alpha, \beta \in M$,
\begin{itemize}
\item $\alpha \mathscr{L} \beta$ if and only if $\im(\alpha) = \im(\beta)$, 

\item $\alpha \mathscr{R} \beta$ if and only if $\dom(\alpha) = \dom(\beta)$, 

\item $\alpha \mathscr{H} \beta $ if and only if $\im(\alpha) = \im(\beta)$ and $\dom(\alpha) = \dom(\beta)$.
\end{itemize}
In $\I(\Omega)$ we also have 
\begin{itemize}
\item $\alpha \mathscr{J} \beta$ if and only if $|\dom(\alpha)| = |\dom(\beta)|$ (if and only if $|\im(\alpha)| = |\im(\beta)|$). 
\end{itemize}

Since $\DPW_n^+$, $\DPW_n^-\cup\DPW_n^+$ and $\DPW_n$ are inverse submonoids of $\I(\Omega_n^0)$,
and $\DPW_n^-$ is an inverse submonoid of $\I_n$, 
it remains to find a description of their Green's relation $\mathscr{J}$. 
Recall that, for a finite monoid, we have $\mathscr{J}=\mathscr{D} \;(=\mathscr{L}\circ\mathscr{R}=\mathscr{R}\circ\mathscr{L}$). 

\medskip 

Next, we fix the following notation. We represent by  
$$
\alpha=\begin{pmatrix}
A_1 & A_2 & \cdots & A_\ell \\
B_1 & B_2 & \cdots & B_\ell
\end{pmatrix} 
$$ 
($0\<\ell\<\lfloor\frac{n}{2}\rfloor$) an element $\alpha\in\DPW_n^-$ such that:
\begin{itemize}
\item[--] $A_1,A_2,\ldots,A_\ell$ are the maximal arcs of $\dom(\alpha)$; 
\item[--] $B_i=A_i\alpha$ for $1\<i\<\ell$. 
\end{itemize}
Notice that $B_1,B_2,\ldots,B_\ell$ must be the maximal arcs of $\im(\alpha)$. 
In this notation, in general, we do not require any order on the maximal arcs $A_1,A_2,\ldots,A_\ell$ of $\dom(\alpha)$. 
However, sometimes it is convenient, without loss of generalization, to consider them ordered by $\min(A_1)<\min(A_2)<\cdots<\min(A_\ell)$. 
Observe that, in this last case, if $1\in\dom(\alpha)$ 
[respectively, $1,n\in\dom(\alpha)$] then $1\in A_1$ [respectively, $1,n\in A_1$].

\smallskip 

Let $1\<i,j,r,s\<n$ be such that $|A_{i,j}|=|A_{r,s}|$. Then, it is easy to check that 
\begin{equation}\label{arcs} 
A_{i,j}g^{r-i}=A_{r,s}=A_{i,j}hg^{s+i-1}
\end{equation}
and, moreover, if $\alpha\in\DPW_n^-$ is such that $A_{i,j}$ is an arc of $\dom(\alpha)$ and $A_{i,j}\alpha=A_{r,s}$ then  
$\alpha|_{A_{i,j}}=g^{r-i}|_{A_{i,j}}$ 
(and, in this case, we say $\alpha$ \textit{preserves the orientation of the arc $A_{i,j}$}) 
or $\alpha|_{A_{i,j}}=hg^{s+i-1}|_{A_{i,j}}$ (and, in this case, we say that $\alpha$ \textit{reverses the orientation of the arc $A_{i,j}$}). 

\smallskip 

Let $\alpha=\left(\begin{smallmatrix}
A_1 & A_2 & \cdots & A_\ell \\
B_1 & B_2 & \cdots & B_\ell
\end{smallmatrix}\right)\in\DPW_n^-$. 
We define the $\mathscr{J}$-\textit{type} of $\alpha$ to be the sequence $(|A_{1\tau}|,|A_{2\tau}|,\ldots,|A_{\ell\tau}|)$, with 
$|A_{1\tau}|\<|A_{2\tau}|\<\cdots\<|A_{\ell\tau}|$ for some permutation $\tau$ of $\{1,\ldots,\ell\}$. 
This notion allows us to give the following description of the Green's relation $\mathscr{J}$ of $\DPW_n^-$. 

\begin{theorem}\label{J-}
Let $\alpha,\beta\in\DPW_n^-$. Then, $\alpha\mathscr{J}\beta$ in $\DPW_n^-$ if and only if $\alpha$ and $\beta$ have the same $\mathscr{J}$-type.
\end{theorem}
\begin{proof}
Let $\alpha=\left(\begin{smallmatrix}
A_1 & A_2 & \cdots & A_\ell \\
B_1 & B_2 & \cdots & B_\ell
\end{smallmatrix}\right)$ 
and
$\beta=\left(\begin{smallmatrix}
A'_1 & A'_2 & \cdots & A'_{\ell'} \\
B'_1 & B'_2 & \cdots & B'_{\ell'}
\end{smallmatrix}\right)$. 
Let $\tau$ and $\tau'$ be permutations of $\{1,\ldots,\ell\}$ and $\{1,\ldots,\ell'\}$, respectively, such that 
$|A_{1\tau}|\<|A_{2\tau}|\<\cdots\<|A_{\ell\tau}|$ and $|A'_{1\tau'}|\<|A'_{2\tau'}|\<\cdots\<|A'_{\ell'\tau'}|$. 

\smallskip 

First, suppose that $\alpha\mathscr{J}\beta$ in $\DPW_n^-$. 
Then, $\alpha\mathscr{J}\beta$ in $\I_n$ and so $\rank(\alpha)=\rank(\beta)$. 
Let $\gamma,\lambda\in\DPW_n^-$ be such that $\beta=\gamma\alpha\lambda$. 
Since $\beta=(\gamma\alpha\alpha^{-1})\alpha\lambda$, 
we may suppose, without loss of generalization, that we took $\gamma\in\DPW_n^-$ with $\rank(\gamma)=\rank(\alpha)=\rank(\beta)$. 
Hence, from $\beta=\gamma\alpha\lambda$, it follows that $\dom(\gamma)=\dom(\beta)$ and $\im(\gamma)=\dom(\alpha)$. 
Thus, $\ell=\ell'$ and 
$\gamma=\left(\begin{smallmatrix}
A'_1 & A'_2 & \cdots & A'_\ell \\
A_{1\xi} & A_{2\xi} & \cdots & A_{\ell\xi}
\end{smallmatrix}\right)$  
 for some permutation $\xi$ of $\{1,\ldots,\ell\}$. 
Therefore, 
$|A'_{1\tau'}|=|A_{1\tau'\xi}|\<|A'_{2\tau'}|=|A_{2\tau'\xi}|\<\cdots\<|A'_{\ell\tau'}|=|A_{\ell\tau'\xi}|$, 
whence $\alpha$ and $\beta$ have the same $\mathcal{J}$-type. 

\smallskip 

Conversely, suppose that $\alpha$ and $\beta$ have the same $\mathcal{J}$-type.  
Then, $\ell=\ell'$ and 
$$
|B_{1\tau}|=|A_{1\tau}|=|A'_{1\tau'}|=|B'_{1\tau'}|\<|B_{2\tau}|=|A_{2\tau}|=|A'_{2\tau'}|=|B'_{2\tau'}|\<\cdots\<|B_{\ell\tau}|=|A_{\ell\tau}|=|A'_{\ell\tau'}|=|B'_{\ell\tau'}|.
$$ 
Hence, by Proposition \ref{char} and (\ref{arcs}) above, we may construct a transformation 
$\lambda=\left(\begin{smallmatrix}
B_{1\tau} & B_{2\tau} & \cdots & B_{\ell\tau} \\
B'_{1\tau'} & B'_{2\tau'} & \cdots & B'_{\ell\tau'}
\end{smallmatrix}\right)\in\DPW_n^-$ such that $\lambda$ preserves the orientation of $B_{i\tau}$ for $1\<i\<\ell$, 
and a transformation 
$\gamma=\left(\begin{smallmatrix}
A'_{1\tau'} & A'_{2\tau'} & \cdots & A'_{\ell\tau'} \\
A_{1\tau} & A_{2\tau} & \cdots & A_{\ell\tau}
\end{smallmatrix}\right)\in\DPW_n^-$ such that, for all $1\<i\<\ell$, $\gamma$ preserves the orientation of $A'_{i\tau'}$ if and only if either 
$\beta$ preserves the orientation of $A'_{i\tau'}$ and $\alpha$ preserves the orientation of $A_{i\tau}$ or 
$\beta$ reverses the orientation of $A'_{i\tau'}$ and $\alpha$ reverses the orientation of $A_{i\tau}$. 
Now, it is a routine matter to check that $\beta=\gamma\alpha\lambda$. 
Similarly, we can construct transformations $\gamma',\lambda'\in\DPW_n^-$ such that $\alpha=\gamma'\beta\lambda'$ and so 
$\alpha\mathscr{J}\beta$ in $\DPW_n^-$, as required.
\end{proof}

Taking into account the isomorphism $\Psi: \DPW_n^+\longrightarrow\DPW_n^-$, as an immediate consequence of Theorem \ref{J-}, we have the following corollary. 

\begin{corollary}\label{J+}
Let $\alpha,\beta\in\DPW_n^+$. Then, $\alpha\mathscr{J}\beta$ in $\DPW_n^+$ if and only if $\alpha\Psi$ and $\beta\Psi$ have the same $\mathscr{J}$-type.
\end{corollary}

Let $\alpha,\beta\in\DPW_n^-\cup\DPW_n^+$ and 
let $\gamma,\lambda\in\DPW_n^-\cup\DPW_n^+$ be such that $\alpha=\gamma\beta\lambda$. 
If $\alpha\in\DPW_n^+$ then, clearly, we have $\gamma,\beta,\lambda\in\DPW_n^+$. 
On the other hand, if $\alpha,\beta\in\DPW_n^-$ then $\alpha=(\gamma\Psi)\beta(\lambda\Psi)$. 
Thus, it is easy to conclude the following result. 

\begin{corollary}\label{J-+}
Let $\alpha,\beta\in\DPW_n^-\cup\DPW_n^+$. Then, $\alpha\mathscr{J}\beta$ in $\DPW_n^-\cup\DPW_n^+$ if and only if 
either $\alpha,\beta\in\DPW_n^-$ and $\alpha$ and $\beta$ have the same $\mathscr{J}$-type 
or 
$\alpha,\beta\in\DPW_n^+$ and $\alpha\Psi$ and $\beta\Psi$ have the same $\mathscr{J}$-type.
\end{corollary}

Now, recall that, by Lemma \ref{split}, if $\alpha\in\DPW_n\setminus(\DPW_n^-\cup\DPW_n^+)$ then $0\in\dom(\alpha)\cup\im(\alpha)$ and $\rank(\alpha)\<4$.  
Thus, for $\alpha,\beta\in\DPW_n$ such that $\rank(\alpha),\rank(\beta)\>5$, we have $\alpha\mathscr{J}\beta$ in $\DPW_n$ if and only if 
$\alpha\mathscr{J}\beta$ in $\DPW_n^-\cup\DPW_n^+$. 
It is also clear that $J_0$ and $J_1$ are $\mathscr{J}$-classes of $\DPW_n$. 

Let $1\<k\<n$ and let $\chi$ be a $\mathscr{J}$-type of an element of $\DPW_n^-$ of rank equal to $k$. 
Then, define 
$$
J_\chi^-=\{\alpha\in J_k\cap\DPW_n^-\mid \mbox{$\chi$ is the $\mathscr{J}$-type of $\alpha$}\}
\quad\text{and}\quad   
J_\chi^+=\{\alpha\in J_{k+1}\cap\DPW_n^+\mid \mbox{$\chi$ is the $\mathscr{J}$-type of $\alpha\Psi$}\}.
$$

It is easy to show that: if $\alpha\in J_2\setminus(\DPW_n^-\cup\DPW_n^+)$ (and so $0\in\dom(\alpha)\cup\im(\alpha)$ with $0\alpha\neq0$) then 
$$
\mbox{$\alpha\mathscr{J}\beta$ in $\DPW_n$ if and only if 
$\beta\in J_2\setminus(\DPW_n^-\cup\DPW_n^+)\cup J_{(2)}^-\cup J_{(1)}^+=J_2\setminus J_{(1,1)}^-$};
$$ 
and, similarly, if $\alpha\in J_3\setminus(\DPW_n^-\cup\DPW_n^+)$ and  $0\not\in\dom(\alpha)\cap\im(\alpha)$ 
then 
$$
\mbox{$\alpha\mathscr{J}\beta$ in $\DPW_n$  if and only if 
$\beta\in \{\gamma\in J_3\setminus(\DPW_n^-\cup\DPW_n^+)\mid 0\not\in\dom(\gamma)\cap\im(\gamma)\}\cup J_{(3)}^-\cup J_{(1,1)}^+$}. 
$$ 
On the other hand, it is a routine matter to show that: if $\alpha\in J_4\setminus(\DPW_n^-\cup\DPW_n^+)$ (and so $0\in\dom(\alpha)\cap\im(\alpha)$) 
then 
$$
\mbox{$\alpha\mathscr{J}\beta$ in $\DPW_n$ if and only if 
$\beta\in J_4\setminus(\DPW_n^-\cup\DPW_n^+)\cup J_{(3)}^+$};
$$ 
and, similarly, if $\alpha\in J_3\setminus(\DPW_n^-\cup\DPW_n^+)$ and  $0\in\dom(\alpha)\cap\im(\alpha)$ 
then 
$$
\mbox{$\alpha\mathscr{J}\beta$ in $\DPW_n$ if and only if 
$\beta\in \{\gamma\in J_3\setminus(\DPW_n^-\cup\DPW_n^+)\mid 0\in\dom(\gamma)\cap\im(\gamma)\}\cup J_{(2)}^+$}. 
$$

Let $J'_4=J_4\setminus(\DPW_n^-\cup\DPW_n^+)\cup J_{(3)}^+$, 
$J'_3=\{\gamma\in J_3\setminus(\DPW_n^-\cup\DPW_n^+)\mid 0\not\in\dom(\gamma)\cap\im(\gamma)\}\cup J_{(3)}^-\cup J_{(1,1)}^+$,
$J''_3=\{\gamma\in J_3\setminus(\DPW_n^-\cup\DPW_n^+)\mid 0\in\dom(\gamma)\cap\im(\gamma)\}\cup J_{(2)}^+$ and 
$J'_2= J_2\setminus(\DPW_n^-\cup\DPW_n^+)\cup J_{(2)}^-\cup J_{(1)}^+$.  
Now, it is easy to conclude the following description of the Green's relation $\mathscr{J}$ of $\DPW_n$. 

\begin{theorem}\label{J}
\begin{enumerate}
\item Let $\alpha,\beta\in J_k$ with $k\>5$.  Then, 
$\alpha\mathscr{J}\beta$ in $\DPW_n$ if and only if 
either $\alpha,\beta\in\DPW_n^-$ and $\alpha$ and $\beta$ have the same $\mathscr{J}$-type 
or 
$\alpha,\beta\in\DPW_n^+$ and $\alpha\Psi$ and $\beta\Psi$ have the same $\mathscr{J}$-type. 

\item $J'_4$, $J_{(4)}^-$, $J_{(1,3)}^-$, $J_{(2,2)}^-$, $J_{(1,1,2)}^-$, $J_{(1,1,1,1)}^-$, 
$J_{(1,2)}^+$ and $J_{(1,1,1)}^+$ are the $\mathscr{J}$-classes of $\DPW_n$ of elements of rank equal to $4$.

\item $J'_3$, $J''_3$, $J_{(1,2)}^-$ and 
$J_{(1,1,1)}^-$ are the $\mathscr{J}$-classes of $\DPW_n$ of elements of rank equal to $3$.

\item $J'_2$ and $J_{(1,1)}^-$  are the $\mathscr{J}$-classes of $\DPW_n$ of elements of rank equal to $2$. 

\item $J_1$ and $J_0$ are the $\mathscr{J}$-classes of $\DPW_n$ of elements of rank equal to $1$ and $0$, respectively. 
\end{enumerate}
\end{theorem}

Notice that, for small $n$'s some sets of the statement of Theorem \ref{J} may not exist. 
Concretely, we have: 
$J_{(1,2)}^-$ and $J_{(1,2)}^+$ only for $n\>5$;  
$J_{(1,1,1)}^-$, $J_{(1,1,1)}^+$, $J_{(1,3)}^-$ and $J_{(2,2)}^-$ only for $n\>6$; 
$J_{(1,1,2)}^-$ only for $n\>7$; and 
$J_{(1,1,1,1)}^-$ only for $n\>8$. 

\section{Generators and rank of $\DPW_n^-$} \label{gr-}

Given a set $\Omega$ and a subset $X$ of $\Omega$, we denote by $\id_X$ the partial identity with domain $X$, 
i.e. $\id_X$ is the partial transformation of $\Omega$ such $\dom(\id_X)=X$ and $(x)\id_X=x$ for $x\in X$. 

\smallskip 

Let us consider the following transformations of $\DPW_n^-$: 
$$
e=\begin{pmatrix}
1&2&\cdots&n-1\\
1&2&\cdots&n-1
\end{pmatrix}
\quad\text{and}\quad
c_j=\left(\begin{array}{cccc|ccc}
1&2&\cdots&j+1&j+3&\cdots&n-1\\ 
1&2&\cdots&j+1&n-1&\cdots&j+3
\end{array}\right) \mbox{~for $1\<j\<\lfloor\frac{n}{2}\rfloor-2$}
$$
(observe that the transformations $c_j$'s only exist for $n\>6$). 
It is also convenient to consider the partial identities $e_i=\id_{\Omega_n\setminus\{i\}}$ for $1\<i\<n$. 
Clearly, $e_1,e_2,\ldots,e_n\in\DPW_n^-$ and $e=e_n$. 
 
\smallskip 
 
Recall that $\DI_n\subseteq\DPW_n^-$ and also that it was proved in \cite{Fernandes&Paulista:2022arxiv} that $\DI_n=\langle g,h,e\rangle$ and, 
moreover, that $\rank(\DI_n)=3$. 
Notice that, we also have $e_1,\ldots,e_n\in\DI_n$ and $e_i=g^{n-i}eg^i$ for $1\<i\<n$. 

\smallskip 

In this section, our first objective is to show that the set $\{g,h,e, c_j\mid 1\<j\<\lfloor\frac{n}{2}\rfloor-2\}$ generates the monoid $\DPW_n^-$. 

If $\alpha\in\DPW_n^-$ is an element of rank $n$, i.e. $\alpha\in J_n^-$, then $\alpha\in\mathscr{D}_{2n}=\langle g,h\rangle$, by Corollary \ref{gpu}. 

Let $\alpha\in\DPW_n^-$ be an element of rank $n-1$. Then, $\dom(\alpha)=A_{i,j}$ and so $\alpha=\alpha|_{A_{i,j}}$,  
for some  $1\<i,j\<n$ (more specifically, $(i,j)\in\{(1,n-1),(2,n)\}$ or $j=i-2$ with $3\<i\<n$). Hence, by Proposition \ref{char}, $\alpha\in\DI_n$.

Combining the two facts just observed, we have the following lemma. 

\begin{lemma}\label{n-1}
Let $\alpha\in\DPW_n^-$ be an element of rank greater than or equal to $n-1$. Then $\alpha\in\langle g,h,e\rangle=\DI_n$. 
\end{lemma}

Given a subset $X$ of $\Omega_n$, an element of $\Omega_n\setminus X$ is called a \textit{gap} of $X$.

\begin{lemma}\label{n-2}
Let $\alpha\in\DPW_n^-$ be an element of rank equal to $n-2$. Then $\alpha\in\langle g,h,e,c_j\,(1\<j\<\lfloor\frac{n}{2}\rfloor-2)\rangle$. 
\end{lemma}
\begin{proof}
Since $\dom(\alpha)$ has exactly two gaps, then $\dom(\alpha)$ has only one maximal arc, if its gaps are \textit{consecutive} 
(i.e. of the form $i$ and $i+1$ for some $1\<i\<n$), and $\dom(\alpha)$ has two (distinct) maximal arcs, otherwise.  

If the gaps of $\dom(\alpha)$ are consecutive, then $\alpha$ is of the form $\left(\begin{smallmatrix}
A_1 \\
B_1 
\end{smallmatrix}\right)$, whence $\alpha=\alpha|_{A_1}$ and so, by Proposition \ref{char}, $\alpha\in\DI_n=\langle g,h,e\rangle$. 

So, suppose that $\dom(\alpha)$ has non consecutive gaps. 
Then $\alpha$ is of the form $\left(\begin{smallmatrix}
A_1 & A_2 \\
B_1 & B_2
\end{smallmatrix}\right)$.
Suppose, without loss of generality, that $|A_1|\<|A_2|$. 
Let $j=|A_1|-1$ and let $1\<r,s\<n$ be such that $A_1=\{r,r+1,\ldots,r+j\}$ and $B_1=\{s,s+1,\ldots,s+j\}$.
Let $\beta=g^{r-1}\alpha g^{n-s+1}$. 
Then  $\alpha=g^{n-r+1}\beta g^{s-1}$ and 
$$
\beta = \begin{pmatrix}
\{1,\ldots,j+1\}&\{j+3,\ldots,n-1\} \\
\{1,\ldots,j+1\}&\{j+3,\ldots,n-1\} 
\end{pmatrix}. 
$$

If $|A_1|=1$ then $A_1=\{r\}$, $B_1=\{s\}$, $j=0$ and 
$
\beta = \left(\begin{smallmatrix}
1&\{3,\ldots,n-1\} \\
1&\{3,\ldots,n-1\} 
\end{smallmatrix}\right).
$
So, 
$
\beta = \left(\begin{smallmatrix}
1&3 & \cdots & n-1 \\
1&3 & \cdots & n-1
\end{smallmatrix}\right)
=e_2e_n=g^{n-2}eg^2e
$ 
or
$
\beta = \left(\begin{smallmatrix}
1&3 & \cdots & n-1\\
1&n-1 & \cdots & 3
\end{smallmatrix}\right) 
=e_2e_nhg=g^{n-2}eg^2ehg
$.
Hence, it follows that $\alpha\in \langle g,h,e\rangle$. 

On the other hand, suppose that $|A_1|>1$. 
Since $|A_1|+|A_2|=n-2$ and $|A_1|\<|A_2|$, we must have $|A_1|\<\lfloor\frac{n}{2}\rfloor-1$. 
Then,  $1\<j=|A_1|-1\<\lfloor\frac{n}{2}\rfloor-2$ and so 
$
\beta = \left(\begin{smallmatrix}
1& \cdots & j+1& j+3 & \cdots & n-1 \\
1& \cdots & j+1& j+3 & \cdots & n-1 
\end{smallmatrix}\right)
=e_{j+2}e_n=g^{n-j-2}eg^{j+2}e
$ 
or 
$
\beta = \left(\begin{smallmatrix}
1& \cdots & j+1& j+3 & \cdots & n-1 \\
1& \cdots & j+1& n-1 & \cdots & j+3 
\end{smallmatrix}\right)
=c_j
$ 
or 
$
\beta = \left(\begin{smallmatrix}
1& \cdots & j+1& j+3 & \cdots & n-1 \\
j+1& \cdots & 1& j+3 & \cdots & n-1 
\end{smallmatrix}\right)
=c_je_{j+2}e_nhg^{j+1}=c_jg^{n-j-2}eg^{j+2}ehg^{j+1}
$ 
or 
$
\beta = \left(\begin{smallmatrix}
1& \cdots & j+1& j+3 & \cdots & n-1 \\
j+1& \cdots & 1& n-1 & \cdots & j+3 
\end{smallmatrix}\right)
=e_{j+2}e_nhg^{j+1}=g^{n-j-2}eg^{j+2}ehg^{j+1}. 
$ 
Hence,  $\alpha\in\langle g,h,e,c_j\,(1\<j\<\lfloor\frac{n}{2}\rfloor-2)\rangle$, as required. 
\end{proof}

\begin{lemma}\label{qid}
Let $\gamma=\left(\begin{smallmatrix}
A_1 & A_2 & \cdots & A_\ell \\
A_1 & A_2 & \cdots & A_\ell
\end{smallmatrix}\right)\in\DPW_n^-$  
be such that  $\gamma|_{\dom(\gamma)\setminus A_j}=\id_{\dom(\gamma)\setminus A_j}$ 
for some $1\<j\<\ell$. 
If $\rank(\gamma)=k\<n-3$ then $\gamma$ is a product of elements of $\DPW_n^-$ with ranks (strictly) greater than $k$.
\end{lemma}
\begin{proof} 
Suppose that $\min(A_1)<\min(A_2)<\cdots<\min(A_\ell)$.

If $\ell=1$ then $\gamma=\gamma|_{A_1}$ and so, by Proposition \ref{char}, $\gamma\in\DI_n=\langle g,h,e\rangle$. 
Hence, $\gamma$ is a product of elements of $\DPW_n^-$ with rank greater than or equal to $n-1$ and $n-1>k$. 

Next, suppose that $\ell=2$. Let $i\in\{1,2\}\setminus\{j\}$. 
Then $A_i=\{r,r+1,\ldots,s\}$, with $r\<s$, or $A_i=\{r,\ldots,n,1,\ldots,s\}$, with $s<r+1$, for some $1\<r,s\<n$. 
Hence $r-1,s+1\not\in\dom(\gamma)$ and, 
since $\dom(\gamma)=A_1\cup A_2$ has $n-k\>3$ gaps, we have $r-2,r-1\not\in\dom(\gamma)$ or $s+1,s+2\not\in\dom(\gamma)$. 
If $r-2\not\in\dom(\gamma)$, take $t=r-1$; otherwise take $t=s+1$. 
Define $\bar\gamma$ to be the extension of $\gamma$ to $t$ defined by $t\bar\gamma=t$. 
Then, clearly, $\bar\gamma=\left(\begin{smallmatrix}
A_j & A_i\cup\{t\} \\
A_j & A_i\cup\{t\}
\end{smallmatrix}\right)\in\DPW_n^-$, 
$\rank(\bar\gamma)=k+1$ and $\gamma=\bar\gamma e_t$.  

Finally, suppose $\ell\>3$. Then, there exists $1\<i\<\ell$ such that $j\not\in\{i,i\oplus1\}$, where 
$$
i\oplus1=\left\{
\begin{array}{ll}
i+1 & \mbox{if $1\<i<\ell$}\\
1 & \mbox{if $i=\ell$}. 
\end{array}
\right. 
$$
Let $\bar A$ be the set of gaps of $\dom(\gamma)$ \textit{between} $A_i$ and $A_{i\oplus1}$, 
i.e. such that $A_i\cup\bar A\cup A_{i\oplus1}$ forms an arc of $\Omega_n$. 
Define $\bar\gamma$ to be the extension of $\gamma$ to $\bar A$ defined by $t\bar\gamma=t$ for $t\in\bar A$. 
Then, it is easy to check that we have $\bar\gamma=\left(\begin{smallmatrix}
A_1&\cdots&A_{(i\oplus1)-2}&A_i\cup\bar A\cup A_{i\oplus1} & A_{i+2} & \cdots &A_\ell\\
A_1&\cdots&A_{(i\oplus1)-2}&A_i\cup\bar A\cup A_{i\oplus1} & A_{i+2} & \cdots &A_\ell 
\end{smallmatrix}\right)\in\DPW_n^-$, 
$\rank(\bar\gamma)=k+|\bar A|>k$ and $\gamma=\bar\gamma\Pi_{t\in \bar A}e_t$, as required. 
\end{proof} 

\begin{lemma}\label{n-k}
Let $\alpha\in\DPW_n^-$ be such that $\rank(\alpha)=k\<n-3$. 
Then $\alpha$ is a product of elements of $\DPW_n^-$ with ranks (strictly) greater than $k$.
\end{lemma}
\begin{proof}
Suppose that 
$\alpha=\left(\begin{smallmatrix}
A_1 & A_2 & \cdots & A_\ell \\
B_1 & B_2 & \cdots & B_\ell
\end{smallmatrix}\right)$, 
with $A_1=A_{r,r'}$ and $B_1=B_{s,s'}$, for some $1\<r,r',s,s'\<n$. Let $t=|A_1|=|B_1|$. 
Then 
$$
g^{r-1}\alpha g^{n-s+1}=
\begin{pmatrix}
A_1g^{n-r+1} & A_2g^{n-r+1} & \cdots & A_\ell g^{n-r+1} \\
B_1g^{n-s+1} & B_2g^{n-s+1} & \cdots & B_\ell g^{n-s+1}
\end{pmatrix}
=\begin{pmatrix}
\{1,\ldots,t\} & A'_2 & \cdots & A'_\ell \\
\{1,\ldots,t\} & B'_2 & \cdots & B'_\ell
\end{pmatrix},
$$ 
with $A'_i=A_ig^{n-r+1}$ and $B'_i=B_ig^{n-s+1}$ for $2\<i\<\ell$.  
Let 
$
\gamma_1=
\left(\begin{smallmatrix}
\{1,\ldots,t\} & B'_2 & \cdots & B'_\ell \\
\{1,\ldots,t\} & B'_2 & \cdots & B'_\ell
\end{smallmatrix}\right)\in\DPW_n^-
$ 
be such that $\gamma_1|_{B'_2 \cup \cdots \cup B'_\ell}=\id_{B'_2 \cup \cdots \cup B'_\ell}$ and $\gamma_1|_{\{1,\ldots,t\}}=g^{r-1}\alpha g^{n-s+1}|_{\{1,\ldots,t\}}$. 
Since $(\gamma_1|_{\{1,\ldots,t\}})^2=\id_{\{1,\ldots,t\}}$, we obtain 
$$
g^{r-1}\alpha g^{n-s+1}\gamma_1=
\begin{pmatrix}
1&\cdots&t & A'_2 & \cdots & A'_\ell \\
1&\cdots&t  & B'_2 & \cdots & B'_\ell
\end{pmatrix}. 
$$
Let $\beta=g^{r-1}\alpha g^{n-s+1}\gamma_1$. 
Then $\rank(\beta)=\rank(\gamma_1)=\rank(\alpha)=k$ and, as $\gamma_1^{-1}=\gamma_1$, 
we get $\alpha=g^{n-r+1}\beta\gamma_1 g^{s-1}$. 

Let $a=t+2$. Then $a\in A'_2$ or $a$ is a gap of $\dom(\beta)$. 
Let $b=\max(B'_2)$. Then $t+2=a\<x\<b$ for all $x\in B'_2$. 
Let
$$
\lambda=\left(\begin{array}{cccccc|cccc}
b+2 & \cdots & n & 1 & \cdots & t &  a & a+1 & \cdots & b \\ 
b+2 & \cdots & n & 1 & \cdots & t & b & b-1 & \cdots & a 
\end{array}\right). 
$$
Clearly, $\lambda\in\DPW_n^-$, $\rank(\lambda)=n-2$ (with $t+1$ and $b+1$ being the gaps of $\dom(\lambda)$ and $\im(\lambda)$) and $\lambda^{-1}=\lambda$. 
Then
$$
\beta\lambda=\begin{pmatrix}
1&\cdots&t & A'_2 & \cdots & A'_\ell \\
1&\cdots&t  & B'_2\lambda & \cdots & B'_\ell\lambda
\end{pmatrix}
$$
(notice that $|B'_i\lambda|=|B'_i|=|A'_i|$ for $2\<i\<\ell$, since the gaps of $\dom(\lambda)$ are also gaps of $\im(\beta)$, 
whence $A'_2,\ldots,A'_\ell$ and $B'_2\lambda, \ldots, B'_\ell\lambda$ are maximal arcs of $\dom(\beta\lambda)$ and $\im(\beta\lambda)$, respectively, 
and $\rank(\beta\lambda)=\rank(\beta)=k$). 

Let $p=|A'_2|=|B'_2|$ and $j=\min(A'_2)-t$. Then $j\>2$, $A'_2=\{t+j,\ldots,t+j+p-1\}$ and $B'_2=\{b-p+1,\ldots,b\}$. 
Hence, $B'_2\lambda=\{a=t+2,\ldots,a+p-1=t+p+1\}$ and so 
$
\beta\lambda=\left(\begin{smallmatrix}
1\:\cdots\:t & \{t+j,\ldots,t+j+p-1\} & A'_3 & \cdots & A'_\ell \\
1\:\cdots\:t & \{t+2,\ldots,t+p+1\} & B'_3\lambda & \cdots & B'_\ell\lambda
\end{smallmatrix}\right) 
$.

Now, let $\gamma_2\in\DPW_n^-$ be such that $\dom(\gamma_2)=\im(\beta\lambda)$, 
 $\gamma_2|_{\im(\beta\lambda)\setminus\{t+2,\ldots, t+p+1\}}=\id_{\im(\beta\lambda)\setminus\{t+2,\ldots, t+p+1\}}$ and
$
\gamma_2|_{\{t+2,\ldots, t+p+1\}} = ((\beta\lambda)|_{\{t+j,\ldots, t+j+p-1\}})^{-1}
\left(\begin{smallmatrix}
t+j & \cdots & t+j+p-1\\
t+2 & \cdots & t+p+1
\end{smallmatrix}\right)
$. 
Observe that, $\rank(\gamma_2)=\rank(\beta\lambda)=k$ and $\gamma_2^{-1}=\gamma_2$. 
Moreover, 
$$
\beta\lambda\gamma_2=
\begin{pmatrix}
1&\cdots&t & t+j&\cdots&t+j+p-1 & A'_3 & \cdots & A'_\ell \\
1&\cdots&t & t+2&\cdots&t+p+1 & B'_3\lambda & \cdots & B'_\ell\lambda
\end{pmatrix}
$$
and $\rank(\beta\lambda\gamma_2)=\rank(\beta\lambda)=k$. 

Next, let 
$$
\delta=\begin{pmatrix}
1 & \cdots & t & t+2 & \cdots & n-j+1\\
1 & \cdots & t & t+j & \cdots & n-1
\end{pmatrix}. 
$$
Clearly, $\delta\in\DPW_n^-$ and, 
since $t+1,\ldots,t+j-1,n$ are all the gaps of $\im(\delta)$ and they are also some of the gaps of $\dom(\beta\lambda\gamma_2)$, 
we get $\rank(\delta)=n-j\>\rank(\beta\lambda\gamma_2)=k$. Moreover, if $j=2$ then $\rank(\delta)=n-2>n-3\>k$, 
i.e. $\rank(\delta)=n-2>k$.  On the other hand, if $j\>3$ then 
$$
\delta= \begin{pmatrix}
1 & \cdots & t & t+2 & \cdots & n-j+2\\
1 & \cdots & t & t+j-1 & \cdots & n-1
\end{pmatrix}
\begin{pmatrix}
1 & \cdots & t+1 & t+j-1 & \cdots & n-2\\
1 & \cdots & t+1 & t+j & \cdots & n-1
\end{pmatrix}, 
$$
i.e. $\delta$ is a product of two transformations of $\DPW_n^-$ with ranks equal to $n-j+1$ and $n-j+1>k$.  
Furthermore, 
$$
\delta\beta\lambda\gamma_2=
\begin{pmatrix}
1&\cdots&t & t+2&\cdots&t+p+1 & A''_3 & \cdots & A''_\ell \\
1&\cdots&t & t+2&\cdots&t+p+1 & B'_3\lambda & \cdots & B'_\ell\lambda
\end{pmatrix}
$$
and $\rank(\delta\beta\lambda\gamma_2)=\rank(\beta\lambda)=k$, 
where $A''_i=A'_i-(j-2)=\{x-j+2\mid x\in A'_i\}$ for $3\<i\<\ell$. 

Now, let $\bar\beta$ be the extension of $\delta\beta\lambda\gamma_2$ to $t+1$ defined by $(t+1)\bar\beta=t+1$. 
Then, clearly, $\bar\beta\in\DPW_n^-$ and $\rank(\bar\beta)=\rank(\delta\beta\lambda\gamma_2)+1=k+1$. 
Moreover, $\delta\beta\lambda\gamma_2=e_{t+1}\bar\beta$ and so, 
taking into account the ranks of the transformations involved, we obtain 
$\beta=\delta^{-1}e_{t+1}\bar\beta\gamma_2^{-1}\lambda^{-1}=\delta^{-1}e_{t+1}\bar\beta\gamma_2\lambda$. 
Hence 
$$
\alpha=g^{n-r+1}\beta\gamma_1 g^{s-1}=g^{n-r+1}  \delta^{-1}e_{t+1}\bar\beta\gamma_2\lambda    \gamma_1 g^{s-1}. 
$$
Since $g,e_{t+1},\bar\beta$ and $\lambda$ have ranks (strictly) greater than $k$, $\delta$ (and so $\delta^{-1}$) is a product of one or two elements of 
$\DPW_n^-$ with ranks (strictly) greater than $k$ and, by Lemma \ref{qid}, $\gamma_1$ and $\gamma_2$ are products of elements of $\DPW_n^-$ with ranks (strictly) greater than $k$, it follows that $\alpha$ is itself a product of elements of $\DPW_n^-$ with ranks (strictly) greater than $k$, as required. 
\end{proof}

In view of Lemma \ref{n-1} and, combining Lemmas \ref{n-2} and \ref{n-k} with an inductive reasoning, it is easy to deduce the following result. 

\begin{proposition}\label{gen-}
$\DPW_n^-=\langle g,h,e,c_j\,(1\<j\<\lfloor\frac{n}{2}\rfloor-2)\rangle$. 
\end{proposition}

The next result, which constitutes our second and last main objective of this section, allows us to conclude that the above set $\{g,h,e, c_j\mid 1\<j\<\lfloor\frac{n}{2}\rfloor-2\}$ of generators of $\DPW_n^-$ has minimum cardinality among the generating sets of $\DPW_n^-$. 

\begin{theorem}\label{r-}
$\rank(\DPW_n^-)=\lfloor\frac{n}{2}\rfloor+1$. 
\end{theorem}
\begin{proof}
By Lemma \ref{n-1}, we have $\{\alpha\in\DPW_n^-\mid\rank(\alpha)\>n-1\}\subseteq\DI_n$. 
Then, since $\DI_n$ has rank equal to three and, clearly, a generating set of $\DI_n$ must have at least two permutations of $\Omega_n$ 
and a transformation with rank $n-1$ (see \cite{Fernandes&Paulista:2022arxiv}), 
also a generating set of $\DPW_n^-$ must at least have two permutations of $\Omega_n$ and a transformation with rank $n-1$.  

Let $\mathcal{X}$ be a set of generators of $\DPW_n^-$.  

Let $1\<j\<\lfloor\frac{n}{2}\rfloor-2$. Then, $\rank(c_j)=n-2$ and $\dom(c_j)$ possesses two maximal arcs, 
one of them with $j+1$ elements and the other with $n-j-3$ elements. 
On the other hand, as $1c_j=1$, $2c_j=2$ and the only element $\sigma\in\D_{2n}$ such that $1\sigma=1$ and $2\sigma=2$ is the identity transformation, $c_j$ cannot be a restriction of some element of $\D_{2n}$ and so $c_j\not\in\DI_n$ (we can also show that 
$\d_{C_n}((j+1)c_j,(j+3)c_j)=\d_{C_n}(j+1,n-1)=j+2>2=\d_{C_n}(j+1,j+3)$, which also allows us to conclude that $c_j\not\in\DI_n$).  

Let $\alpha_1,\alpha_2,\ldots,\alpha_m\in\mathcal{X}$ ($m\>1$) be such that $c_j=\alpha_1\alpha_2\cdots\alpha_m$. 
As $\{\alpha\in\DPW_n^-\mid\rank(\alpha)\>n-1\}\subseteq\DI_n$, $c_j\not\in\DI_n$ and $\rank(c_j)=n-2$, 
there exists $1\<i\<m$ such that $\rank(\alpha_i)=n-2$. 
Now, we have $\dom(c_j)=\dom(\alpha_1\alpha_2\cdots\alpha_m)\subseteq\dom(\alpha_1\alpha_2\cdots\alpha_i)$ and, since 
$$
n-2=\rank(c_j)=\rank(\alpha_1\alpha_2\cdots\alpha_m)\<\rank(\alpha_1\alpha_2\cdots\alpha_i)\<\rank(\alpha_i)=n-2, 
$$
we get $\rank(\alpha_1\alpha_2\cdots\alpha_i)=n-2=\rank(c_j)$ and so we can conclude that $\dom(c_j)=\dom(\alpha_1\alpha_2\cdots\alpha_i)$. 
Hence, $|\{1,\ldots,j+1\}\alpha_1\alpha_2\cdots\alpha_i|=j+1$, 
$|\{j+3,\ldots,n-1\}\alpha_1\alpha_2\cdots\alpha_i|=n-j-3$ and 
$\{1,\ldots,j+1\}\alpha_1\alpha_2\cdots\alpha_i$ and  $\{j+3,\ldots,n-1\}\alpha_1\alpha_2\cdots\alpha_i$ 
are the maximal arcs of $\im(\alpha_1\alpha_2\cdots\alpha_i)$. 
On the other hand, $\im(\alpha_1\alpha_2\cdots\alpha_i)\subseteq\im(\alpha_i)$ and $\rank(\alpha_1\alpha_2\cdots\alpha_i)=n-2=\rank(\alpha_i)$
imply that  $\im(\alpha_1\alpha_2\cdots\alpha_i)=\im(\alpha_i)$.
Thus, $\{1,\ldots,j+1\}\alpha_1\alpha_2\cdots\alpha_i$ and  $\{j+3,\ldots,n-1\}\alpha_1\alpha_2\cdots\alpha_i$ 
are also the maximal arcs of $\im(\alpha_i)$. 

Therefore, we proved that, for each $1\<j\<\lfloor\frac{n}{2}\rfloor-2$, $\mathcal{X}$ possesses an element of rank $n-2$ with a maximal arc  
 with $j+1$ elements and another maximal arc with $n-j-3$ elements. 
 Since $j+1\<\lfloor\frac{n}{2}\rfloor-1\<\lfloor\frac{n+1}{2}\rfloor-1\<n-j-3$, 
 we may conclude that $\mathcal{X}$ possesses at least $\lfloor\frac{n}{2}\rfloor-2$ distinct elements of rank $n-2$ and so 
 $|\mathcal{X}|\>\lfloor\frac{n}{2}\rfloor+1$. 
 
 Finally, as the generating set $\{g,h,e, c_j\mid 1\<j\<\lfloor\frac{n}{2}\rfloor-2\}$ of $\DPW_n^-$ has exactly $\lfloor\frac{n}{2}\rfloor+1$ elements, 
 it follows that $\rank(\DPW_n^-)=\lfloor\frac{n}{2}\rfloor+1$, as required. 
\end{proof}

\section{Generators and rank of $\DPW_n$} \label{gr}

Let us consider the following transformations of $\DPW_n^+$: 
$$
g_0=\begin{pmatrix}
0&1&2&\cdots&n-1&n\\
0&2&3&\cdots&n&1
\end{pmatrix}
,\quad 
h_0=\begin{pmatrix}
0&1&2&\cdots&n-1&n\\
0&n&n-1&\cdots&2&1
\end{pmatrix}, 
$$
$$
e_0=\begin{pmatrix}
0&1&2&\cdots&n-1\\
0&1&2&\cdots&n-1
\end{pmatrix}
\quad\text{and}\quad
b_j=\left(\begin{array}{cccccccc}
0&1&2&\cdots&j+1&j+3&\cdots&n-1\\ 
0&1&2&\cdots&j+1&n-1&\cdots&j+3
\end{array}\right) \mbox{~for $1\<j\<\lfloor\frac{n}{2}\rfloor-2$.}
$$
Then, since $\Psi:\DPW_n^+\longrightarrow\DPW_n^-$ is an isomorphism, in view of Proposition \ref{gen-} and Theorem \ref{r-}, 
we immediately have the following corollary. 

\begin{corollary}\label{genr+}
$\DPW_n^+=\langle g_0,h_0,e_0,b_j\,(1\<j\<\lfloor\frac{n}{2}\rfloor-2)\rangle$ and $\rank(\DPW_n^+)=\lfloor\frac{n}{2}\rfloor+1$.
\end{corollary}

Next, let us also consider the transformation  
$$
\iota=\begin{pmatrix}
1&2&\cdots&n\\
1&2&\cdots&n
\end{pmatrix} \in\DPW_n^-.
$$
Then, we have $g=g_0\iota$, $h=h_0\iota$, $e=e_0\iota$ and $c_j=b_j\iota$, for $1\<j\<\lfloor\frac{n}{2}\rfloor-2$, and it is easy to deduce the following result. 

\begin{corollary}\label{genr-+}
$\DPW_n^-\cup\DPW_n^+=\langle g_0,h_0,e_0,\iota,b_j\,(1\<j\<\lfloor\frac{n}{2}\rfloor-2)\rangle$ and $\rank(\DPW_n^-\cup\DPW_n^+)=\lfloor\frac{n}{2}\rfloor+2$.
\end{corollary}

Finally, consider the transformation 
$$
z=\begin{pmatrix}
0&1&2&n\\
1&0&2&n
\end{pmatrix}\in\DPW_n.  
$$
Observe that $z\not\in\DPW_n^-\cup\DPW_n^+$.

Now, we can prove the following proposition. 

\begin{proposition}\label{gen}
$\DPW_n=\langle g_0,h_0,e_0,\iota,z,b_j\,(1\<j\<\lfloor\frac{n}{2}\rfloor-2)\rangle$. 
\end{proposition}
\begin{proof}
Let $\alpha\in\DPW_n$. If $\alpha\in\DPW_n^-\cup\DPW_n^+$ then, by Corollary \ref{genr-+}, 
$\alpha\in \langle g_0,h_0,e_0,\iota,z,b_j\,(1\<j\<\lfloor\frac{n}{2}\rfloor-2)\rangle$. 
Then, suppose that $\alpha\in\DPW_n\setminus(\DPW_n^-\cup\DPW_n^+)$. 
Hence, $0\in\dom(\alpha)\cup\im(\alpha)$ and, by Lemma \ref{split}, we have $\rank(\alpha)\<4$.  
In what follows, we prove that $\alpha\in \langle g_0,h_0,e_0,\iota,z,b_j\,(1\<j\<\lfloor\frac{n}{2}\rfloor-2)\rangle$, 
by considering all possible cases for $\alpha$. 

\smallskip 

{\sc case} 1. Suppose that $\rank(\alpha)=4$. Then, by Property 3 of Lemma \ref{split}, $0\in\dom(\alpha)\cap\im(\alpha)$ and $0\alpha\neq0$. 
Hence 
$
\alpha=\left(\begin{smallmatrix}
0&i&a&b\\
j&0&c&d
\end{smallmatrix}\right)
$
for some $i,j,a,b,c,d\in\Omega_n$. 
Since $\d(0,c)=\d(0,d)=\d(0,a)=\d(0,b)=1$, then $\d(i,a)=\d(i,b)=\d(j,c)=\d(j,d)=1$, 
whence $\{a,b\}=\{i-1,i+1\}$ and $\{c,d\}=\{j-1,j+1\}$. 
So,    
$
\alpha=\left(\begin{smallmatrix}
0&i-1&i&i+1\\
j&j\pm1&0&j\mp1
\end{smallmatrix}\right)
$ 
and we obtain 
$
\alpha=\alpha_1 z\alpha_2$, 
with 
$
\alpha_1=\left(\begin{smallmatrix}
0&i-1&i&i+1\\
0&n&1&2
\end{smallmatrix}\right), 
\alpha_2= 
\left(\begin{smallmatrix}
0&1&2&n\\
0&j& j\mp1&j\pm1
\end{smallmatrix}\right)
\in\DPW_n^+$.  
By Corollary \ref{genr+}, it follows that $\alpha\in \langle g_0,h_0,e_0,\iota,z,b_j\,(1\<j\<\lfloor\frac{n}{2}\rfloor-2)\rangle$. 

\smallskip 

{\sc case} 2. Suppose that $\rank(\alpha)=3$ and $0\in\dom(\alpha)\cap\im(\alpha)$. Then, as $0\alpha\neq0$, we have 
$
\alpha=\left(\begin{smallmatrix}
0&i&a\\
j&0&b 
\end{smallmatrix}\right)
$
for some $i,j,a,b\in\Omega_n$. 
Let 
$
\alpha_1=\left(\begin{smallmatrix}
0&i&a\\
0&1&2
\end{smallmatrix}\right)
$
and 
$
\alpha_2= 
\left(\begin{smallmatrix}
0&1&2\\
0&j&b
\end{smallmatrix}\right)
$. 
Then, as  $\d(i,a)=\d(0,b)=1=\d(1,2)$ and  $\d(j,b)=\d(0,a)=1=\d(1,2)$, 
we have $\alpha_1,\alpha_2\in\DPW_n^+$. 
Since $\alpha=\alpha_1 z\alpha_2$, by Corollary \ref{genr+}, 
we have $\alpha\in \langle g_0,h_0,e_0,\iota,z,b_j\,(1\<j\<\lfloor\frac{n}{2}\rfloor-2)\rangle$. 

\smallskip 

{\sc case} 3. Suppose that $\rank(\alpha)=3$ and $0\not\in\dom(\alpha)\cap\im(\alpha)$. 
Then, we must have $0\in\dom(\alpha)\setminus\im(\alpha)$ or $0\in\im(\alpha)\setminus\dom(\alpha)$. 

First, suppose that $0\in\dom(\alpha)\setminus\im(\alpha)$. Then,  
$
\alpha=\left(\begin{smallmatrix}
0&a&b\\
i&c&d 
\end{smallmatrix}\right)
$
for some $i,a,b,c,d\in\Omega_n$. 
Since $\d(0,a)=1=\d(0,b)$, then $\d(i,c)=1=\d(i,d)$, whence $\d(c,d)=2$ and so $\d(a,b)=2$. 
Then, 
$
\alpha_1=\left(\begin{smallmatrix}
0&a&b\\
0&2&n
\end{smallmatrix}\right)\in\DPW_n^+ 
$
and
$
\alpha_2= 
\left(\begin{smallmatrix}
1&2&n\\
i&c&d
\end{smallmatrix}\right)\in\DPW_n^-. 
$ 
Now, as $\alpha=\alpha_1 z\alpha_2$, by Corollary \ref{genr-+}, 
we have $\alpha\in \langle g_0,h_0,e_0,\iota,z,b_j\,(1\<j\<\lfloor\frac{n}{2}\rfloor-2)\rangle$. 

Secondly, suppose that $0\in\im(\alpha)\setminus\dom(\alpha)$. Then, $0\in\dom(\alpha^{-1})\setminus\im(\alpha^{-1})$ and we also have 
$\rank(\alpha^{-1})=3$ and $0\not\in\dom(\alpha^{-1})\cap\im(\alpha^{-1})$. 
Hence, by the first subcase, we get $\alpha^{-1}=\alpha_1 z\alpha_2$, for some $\alpha_1\in \DPW_n^+$ and  $\alpha_2\in\DPW_n^-$. 
As $z^{-1}=z$, we obtain $\alpha=\alpha_2^{-1} z\alpha_1^{-1}$. 
Since $\alpha_1^{-1}\in \DPW_n^+$ and  $\alpha_2^{-1}\in\DPW_n^-$, by Corollary \ref{genr-+}, 
we have $\alpha\in \langle g_0,h_0,e_0,\iota,z,b_j\,(1\<j\<\lfloor\frac{n}{2}\rfloor-2)\rangle$. 

\smallskip 

{\sc case} 4. Suppose that $\rank(\alpha)=2$.  Then $\alpha$ is of one of the following forms: 
\begin{description}
\item[--] $
\alpha=\left(\begin{smallmatrix}
0&i\\
j&0
\end{smallmatrix}\right)
$
for some $i,j\in\Omega_n$.  In  this subcase, we have 
$
\alpha=\left(\begin{smallmatrix}
0&i\\
0&1
\end{smallmatrix}\right)
z
\left(\begin{smallmatrix}
1&0\\
j&0
\end{smallmatrix}\right)
$
and 
$
\left(\begin{smallmatrix}
0&i\\
0&1
\end{smallmatrix}\right),
\left(\begin{smallmatrix}
1&0\\
j&0
\end{smallmatrix}\right)\in\DPW_n^+
$, 
whence $\alpha\in \langle g_0,h_0,e_0,\iota,z,b_j\,(1\<j\<\lfloor\frac{n}{2}\rfloor-2)\rangle$, by Corollary \ref{genr+};  

\item[--] $
\alpha=\left(\begin{smallmatrix}
0&i\\
a&b
\end{smallmatrix}\right)
$
for some $i,a,b\in\Omega_n$.  In  this subcase, we have 
$
\alpha=\left(\begin{smallmatrix}
0&i\\
0&2
\end{smallmatrix}\right)
z
\left(\begin{smallmatrix}
1&2\\
a&b
\end{smallmatrix}\right)
$, 
$
\left(\begin{smallmatrix}
0&i\\
0&2
\end{smallmatrix}\right)\in\DPW_n^+ 
$ 
and, 
as $\d(1,2)=1=\d(0,i)=\d(a,b)$, 
$
\left(\begin{smallmatrix}
1&2\\
a&b
\end{smallmatrix}\right)\in\DPW_n^-
$, 
whence $\alpha\in \langle g_0,h_0,e_0,\iota,z,b_j\,(1\<j\<\lfloor\frac{n}{2}\rfloor-2)\rangle$, by Corollary \ref{genr-+};  

\item[--] $
\alpha=\left(\begin{smallmatrix}
a&b\\
0&i
\end{smallmatrix}\right)
$
for some $i,a,b\in\Omega_n$.  
Then,  
$
\alpha^{-1}=\left(\begin{smallmatrix}
0&i\\
a&b
\end{smallmatrix}\right)
$
and, by the previous subcase, 
we have $\alpha^{-1}=
\left(\begin{smallmatrix}
0&i\\
0&2
\end{smallmatrix}\right)
z
\left(\begin{smallmatrix}
1&2\\
a&b
\end{smallmatrix}\right)
$,
with 
$
\left(\begin{smallmatrix}
1&2\\
a&b
\end{smallmatrix}\right)\in\DPW_n^-
$,
and so 
$\alpha=
\left(\begin{smallmatrix}
a&b\\
1&2
\end{smallmatrix}\right)
z
\left(\begin{smallmatrix}
0&2\\
0&i
\end{smallmatrix}\right)
$. 
Since 
$
\left(\begin{smallmatrix}
a&b\\
1&2
\end{smallmatrix}\right)\in\DPW_n^-
$
and
$
\left(\begin{smallmatrix}
0&2\\
0&i
\end{smallmatrix}\right)\in\DPW_n^+
$, 
by Corollary \ref{genr-+}, 
we have $\alpha\in \langle g_0,h_0,e_0,\iota,z,b_j\,(1\<j\<\lfloor\frac{n}{2}\rfloor-2)\rangle$. 

\end{description}

\smallskip 

{\sc case} 5. Finally, suppose that $\rank(\alpha)=1$.  Then, for some $i\in\Omega_n$, 
$
\alpha=\left(\begin{smallmatrix}
0\\
i
\end{smallmatrix}\right)=
z
\left(\begin{smallmatrix}
1\\
i
\end{smallmatrix}\right)
$
or 
$
\alpha=\left(\begin{smallmatrix}
i\\
0
\end{smallmatrix}\right)=
\left(\begin{smallmatrix}
i\\
1
\end{smallmatrix}\right)
z
$.
Since 
$
\left(\begin{smallmatrix}
1\\
i
\end{smallmatrix}\right),
\left(\begin{smallmatrix}
i\\
1
\end{smallmatrix}\right)\in\DPW_n^-
$, 
by Proposition \ref{gen-}, 
in both subcases, we get $\alpha\in \langle g_0,h_0,e_0,\iota,z,b_j\,(1\<j\<\lfloor\frac{n}{2}\rfloor-2)\rangle$, 
as required. 
\end{proof}

In particular, by Proposition \ref{gen}, we have $\DPW_4=\langle g_0,h_0,e_0,\iota,z\rangle$. Since $e_0=g_0^3z^2g_0$ for $n=4$, 
then $\DPW_4=\langle g_0,h_0,\iota,z\rangle$. On the other hand, any generating set of $\DPW_4$ must have at least: two permutations of $\Omega_n^0$ (in order to generate its group of units which is isomorphic to a dihedral group of order $2\times4$), one element of $\DPW_4^-$ and one element of 
$\DPW_4\setminus(\DPW_4^-\cup\DPW_4^+)$. Thus, $\{g_0,h_0,\iota,z\}$ is a generating set of minimal size of $\DPW_4^-$ 
and so $\rank(\DPW_4)=4$. 

Next, notice that $\rank(g_0)=\rank(h_0)=n+1$, $\rank(e_0)=n$, $\rank(\iota)=n$, $\rank(z)=4$ and $\rank(b_j)=n-1$, 
for $1\<j\<\lfloor\frac{n}{2}\rfloor-2$ (such a $j$ only exists for $n\>6$). So, for $n\geqslant5$, all the elements of the generating set 
$\{g_0,h_0,e_0,\iota,b_j\mid 1\<j\<\lfloor\frac{n}{2}\rfloor-2\}$ of $\DPW_n^-\cup\DPW_n^+$ have ranks greater than or equal to $5$ 
(notice that, for $n=5$ we have no $b_j$'s). 
On the other hand, as observed before, all elements of $\DPW_n\setminus(\DPW_n^-\cup\DPW_n^+)$ have ranks less than or equal to $4$. 
Thus, any generating set of $\DPW_n$ must contain a generating set of $\DPW_n^-\cup\DPW_n^+$ and so 
$\rank(\DPW_n)>\rank(\DPW_n^-\cup\DPW_n^+)=\lfloor\frac{n}{2}\rfloor+2$. 
This fact, in view of the Proposition \ref{gen}, demonstrates our main result that we state below and with which we end this paper.

\begin{theorem}\label{rank} 
$\rank(\DPW_4)=4$ and 
$\rank(\DPW_n)=\lfloor\frac{n}{2}\rfloor+3$, for $n\geqslant5$. 
\end{theorem}

\subsection*{Acknowledgments} 
The author would like to thank Ricardo Guilherme for pointing out several typos in the first version of this paper.


\bigskip

\lastpage

\end{document}